\documentclass[12pt]{article}
\usepackage{amsmath,amsthm,amssymb}

\usepackage{graphicx}
\usepackage{subfigure}
\usepackage[usenames]{color}
\usepackage{hyperref}

\begingroup\makeatletter\ifx\SetFigFont\undefined%
\gdef\SetFigFont#1#2#3#4#5{%
  \reset@font\fontsize{#1}{#2pt}%
  \fontfamily{#3}\fontseries{#4}\fontshape{#5}%
  \selectfont}%
\fi\endgroup%

\newcommand{\E}{{\mathbb{\,E}}}
\newcommand{\I}{{\mathbb{\,I}}}
\newcommand{\ac}{{\mathcal{A}}}
\newcommand{\bc}{{\mathcal{B}}}
\newcommand{\ec}{{\mathcal{E}}}
\newcommand{\gc}{{\mathcal{G}}}
\newcommand{\hc}{{\mathcal{H}}}
\newcommand{\Sc}{{\mathcal{S}}}
\newcommand{\Var} {\operatorname{Var}}
\newcommand{\Prob} {\mathbb {P}}
\newcommand{\prob}[1]{\Prob\left(#1\right)}
\newcommand{\Pc}[2]{\prob{#1\,|\,#2}}

\newcommand{\eps} {\varepsilon}
\newcommand{\fee} {\varphi}
\newcommand{\R} {\mathbb R}
\newcommand{\Z} {\mathbb Z}

\newcommand{\hh} {\mathbb H}
\newcommand{\e}{\mathrm{e}}
\newcommand{\Bi}{\operatorname{Bi}}
\newcommand{\X}{\mathbf{X}}
\newcommand{\x}{\mathbf{x}}
\newcommand{\Y}{\mathbf{Y}}
\newcommand{\y}{\mathbf{y}}
\newcommand{\z}{\mathbf{z}}
\newcommand{\green}{\operatorname{green}}
\newcommand{\degb}{\deg_{\green}}
\newcommand{\red}{\operatorname{red}}
\newcommand{\sw}{\operatorname{Sw}}
\newcommand{\Hn}[1]{\hh^{(k)}(n,#1)}
\newcommand{\Hnd}{\Hn{d}}
\newcommand{\THnd}{\tilde \hh^{(k)}(n,d)}
\newcommand{\Hnm}{\Hn{m}}
\newcommand{\Hnp}{\Hn{p}}
\newcommand{\hcnd}{\hc^{(k)}(n,d)}
\newcommand{\PHn}[1]{\hh_*^{(k)}(n,#1)}
\newcommand{\PHnd}{\PHn{d}}
\newcommand{\tvd}{\operatorname{d}_{TV}}
\newcommand{\hyp}{\mathsf H}

\newtheorem{theorem}{Theorem}
\newtheorem{lemma}[theorem]{Lemma}
\newtheorem{corollary}[theorem]{Corollary}
\newtheorem{proposition}[theorem]{Proposition}
\newtheorem{conjecture}{Conjecture}
\theoremstyle{remark}
\newtheorem{remark}{Remark}

\title{Loose Hamilton Cycles in Regular Hypergraphs\thanks{MSC2010 codes: 05C65, 05C80, 05C45.}}

\author{\normalsize ANDRZEJ DUDEK{$^1$}\thanks{Supported in part by Simons Foundation Grant \#244712.} \and \normalsize ALAN FRIEZE{$^{2}$}\thanks{Supported in part by NSF grant {CCF1013110}.} \and
\normalsize ANDRZEJ RUCI\'NSKI{$^{3}$}\thanks{Research supported by the Polish NSC grant N201 604940. Part of research performed at Emory University, Atlanta.} \and
\normalsize MATAS \v{S}ILEIKIS{$^{4}$}\thanks{Research supported by the Polish NSC grant N201 604940. Part of research performed at Adam Mickiewicz University, Pozna\'n.} \\
\small {$^1$}Department of Mathematics, Western Michigan University, Kalamazoo, MI \\
\small \texttt{andrzej.dudek@wmich.edu} \\
\small {$^2$}Department of Mathematical Sciences, Carnegie Mellon University, Pittsburgh, PA\\
\small \texttt{alan@random.math.cmu.edu} \\
\small {$^3$}Department of Discrete Mathematics, Adam Mickiewicz University, Pozna\'n, Poland\\
\small \texttt{rucinski@amu.edu.pl} \\
\small {$^4$}Department of Mathematics, Uppsala University, Sweden\\
\small \texttt{matas.sileikis@math.uu.se}
}

\begin{document}
\date{}
  \maketitle
  \begin{abstract}
    We establish a relation between two uniform models of random $k$-graphs (for constant $k \ge 3$) on $n$ labeled vertices: $\Hnm$, the random $k$-graph with exactly $m$ edges, and $\Hnd$, the random $d$-regular $k$-graph. By extending to $k$-graphs the switching technique of McKay and Wormald, we show that, for some range of $d=d(n)$ and a constant $c>0$, if $m \sim cnd$, then one can couple $\Hnm$ and $\Hnd$ so that the latter contains the former  with proba\-bility tending to one as $n\to\infty$. In view of known results on the existence of a loose Hamilton cycle in $\Hnm$, we conclude that $\Hnd$ contains a loose Hamilton cycle when $d \gg \log n$ (or just $d \ge C \log n$, if $k = 3$) and $d = o(n^{1/2})$.
  \end{abstract}
  \section{Introduction}
  A \emph{$k$-uniform hypergraph} (or \emph{$k$-graph} for short) on a vertex set $V = \{1, \dots, n\}$ is a family of $k$-element subsets of $V$.
  A $k$-graph $H = (V,E)$ is $d$-regular, if the degree of every vertex is $d$:
  \begin{equation*}
    \deg(v) := |\left\{ e \in E : v \in e \right\}| = d, \quad v = 1, \dots, n.
  \end{equation*}
  Let $\hcnd$ be the family of all such graphs.
  Further we tacitly assume that $k$ divides $nd$.
  By $\Hnd$ we denote the \emph{regular} random graph, which is chosen uniformly at random from $\hcnd$. Let
  \begin{equation*}
    M := nd/k
  \end{equation*}
  stand for the number of edges of $\Hnd$.

  Let us recall two more standard models of random $k$-graphs on $n$ vertices. For $p \in [0,1]$, the \emph{binomial} random $k$-graph $\Hnp$ is a random $k$-graph obtained by including every of the $\binom n k$ possible edges with probability $p$ independently of others. For integer $m \in [0,\binom n k]$, the \emph{uniform} random graph $\Hnm$ is chosen uniformly at random among $k$-graphs with precisely $m$ edges.

  We study the behavior of random $k$-graphs as $n \to \infty$. Parameters $d, m, p$ are treated as functions of $n$. We use the asymptotic notation $O(\cdot), o(\cdot), \Theta(\cdot), \sim$ (as it is defined in, say, \cite{K}), with respect to $n$ tending to infinity and assume that implied constants may depend on $k$.
  Given a sequence of events $(\ac_n)$, we say that $\ac_n$ happens \emph{asymptotically almost surely} (\emph{a.a.s.}) if $\prob{\ac_n} \to 1$, as $n \to \infty$.

  The main result of the paper is that we can couple $\Hnd$ and $\Hnm$ so that the latter is a subgraph of the former a.a.s.\
\begin{theorem}
  \label{thm:main}
  For every $k \ge 3$, there are positive constants $c$ and $C$ such that if $d \ge C \log n$, $d = o(n^{1/2})$ and $m = \lfloor cM \rfloor = \lfloor cnd/k \rfloor$, then one can define a joint distribution of random graphs $\Hnd$ and $\Hnm$ in such a way that
  \begin{equation*}\label{eq:embedding}
    \Hnm \subset \Hnd \qquad \text{a.a.s.\ }
  \end{equation*}
\end{theorem}
\noindent To prove Theorem \ref{thm:main}, we consider a generalization of a $k$-graph that allows loops and multiple edges.
By a \emph{$k$-multigraph} on the vertex set $[n]$ we mean a multiset of $k$-element multisubsets of $[n]$. An edge is called  a \emph{loop} if it contains more than one copy of some vertex and otherwise it is called \emph{a~proper edge}. 

The idea of the proof and the structure of the paper are as follows. In Section~\ref{sec:prelim} we generate two models of random $k$-multigraphs by drawing random sequences from~$[n]$ and cutting them into consecutive segments of length $k$. By accepting an edge only if it is not a loop and does not coincide with a previously accepted edge, after $m$ successful trials we obtain $\Hnm$. On the other hand, by allowing $d$ copies of each vertex, and accepting every edge, after $dn/k$ steps we obtain a $d$-regular $k$\mbox{-}multigraph $\PHnd$. Then we show that $\PHnd$ a.a.s.\  has no multiple edges and relatively few loops.
In Section \ref{sec:embedding} we couple the two random processes in such a way that $\Hnm$ is a.a.s.\  contained in an initial segment of $\PHnd$, which we call \emph{red}.
In Section \ref{sec:redloops} we eliminate at once all red loops of $\PHnd$ by swapping them with randomly selected non-red (\emph{green}) proper edges. Finally, in Section \ref{sec:greenloops}, we eliminate the green loops one by one using a certain random procedure (called switching) which does not destroy the previously embedded copy of $\Hnm$ and, at the same time, transforms $\PHnd$ into a $k$-graph $\THnd$, which is distributed approximately as $\Hnd$, that is, almost uniformly. Theorem \ref{thm:main} follows by a (maximal) coupling of $\THnd$ and $\Hnd$.

A consequence of Theorem \ref{thm:main} is that $\Hnd$ inherits from $\Hnm$ properties that are increasing, that is to say, properties that are preserved as new edges are added. An example of such a property is hamiltonicity, that is, containment of a Hamilton cycle.

A \emph{loose Hamilton cycle} on $n$ vertices is a set of edges $e_1, \dots, e_l$ such that for some cyclic order of the vertices every edge $e_i$ consists of $k$ consecutive vertices, and $|e_i \cap e_{i+1}| = 1$ for every $i \in [l]$, where $e_{l+1} := e_1$. A necessary condition for the existence of a loose Hamilton cycle on $n$ vertices is $(k-1)|n$, which we will assume whenever relevant.

  The history of hamiltonicity of regular graphs is rich and exciting (see \cite{W99}). However, we state only the final results here. Asymptotic hamiltonicity was proved by Robinson and Wormald \cite{RW94} in 1994 for any fixed $d \ge 3$, by Krivelevich, Sudakov, Vu and Wormald \cite{KSVW} in 2001 for $d \ge n^{1/2} \log n$, and by Cooper, Frieze and Reed \cite{CFR} in~2002 for $C \le d \le n/C$ and some large constant $C$.

  The threshold for existence of a loose Hamilton cycle in $\Hnp$ was determined by Frieze \cite{F} (for $k=3$) as well as Dudek and Frieze \cite{DF} (for $k \ge 4$) under a divisibility condition $2(k-1)|n$, which was relaxed to $(k -1)| n$ by Dudek, Frieze, Loh and Speiss~\cite{DFLS}.

  However, we formulate these results for the model $\Hnm$, such a possibility being provided to us by the asymptotic equivalence of models $\Hnp$ and $\Hnm$ (see, e.g., Corollary 1.16 in \cite{JLR}).
  \begin{theorem}[\cite{F}, \cite{DFLS}]
    \label{thm:hamHnm3}
    There is a constant $C > 0$ such that if $m \ge Cn\log n$, then
  \begin{equation*}
    \hh^{(3)}(n,m) \text{ contains a loose Hamilton cycle a.a.s.\ }
  \end{equation*}
  \end{theorem}
  \begin{theorem}[\cite{DF}, \cite{DFLS}]
    \label{thm:hamHnmk}
    Let $k \ge 4$. If $n \log n = o(m)$, then
  \begin{equation*}
    \Hnm \text{ contains a loose Hamilton cycle a.a.s.\ }
  \end{equation*}
  \end{theorem}
  \noindent Theorems \ref{thm:main}, \ref{thm:hamHnm3}, and \ref{thm:hamHnmk} immediately imply the following fact.
  \begin{corollary}
    \label{cor:ham}
    There is a constant $C > 0$ such that if $C\log n \le d = o(n^{1/2})$, then
  \begin{equation*}
    \hh^{(3)}(n,d) \text{ contains a loose Hamilton cycle a.a.s.\ }
  \end{equation*}
  For every $k \ge 4$ if $\log n = o(d)$ and $d = o(n^{1/2})$, then
  \begin{equation*}
    \Hnd \text{ contains a loose Hamilton cycle a.a.s.\ }
  \end{equation*}
  \end{corollary}

  \section{Preliminaries}\label{sec:prelim}
We say that a $k$-multigraph is \emph{simple} if it is a $k$-graph, that is, if it contains neither multiple edges nor loops.

Given a sequence $\x \in [n]^{ks}$, $s \in \{ 1, 2, \dots \}$, let $\hyp(\x)$ stand for a $k$-multi\-graph with~$s$ edges
\begin{equation*}
   x_{ki+1}\dots x_{ki + k}, \qquad i = 0, \dots, s-1.
\end{equation*}
   In what follows it will be convenient to work directly with the sequence $\x$ rather than with the $k$-multigraph $\hyp(\x)$. Recycling the notation, we still refer to the $k$-tuples of $\x$ which correspond to the edges, loops, and proper edges of $\hyp(\x)$ as \emph{edges}, \emph{loops}, and \emph{proper edges} of $\x$, respectively. We say that $\x$ contains \emph{multiple edges}, if $\hyp(\x)$ contains multiple edges, that is, some two edges of $\x$ are identical as multisets. By $\lambda(\x)$ we denote the number of loops in $\x$.

Let $\X = (X_1, \dots, X_{nd})$ be a sequence of i.i.d.\;random variables, each distributed uniformly over $[n]$:
\begin{equation*}
  \prob{X_i = j} = \frac{1}{n}, \qquad 1 \le i \le nd, \quad 1 \le j \le n.
\end{equation*}
Set
\begin{equation*}
  L := n^{1/4}d^{1/2}.
\end{equation*}
  \begin{proposition}
    \label{prop:XE}
    If $d \to \infty$, and $d = o(n^{1/2})$, then a.a.s.\
      $\X$ has no multiple edges and $\lambda(\X) \le L$.
  \end{proposition}
  \begin{proof}
    Both statements hold a.a.s.\  by Markov's inequality, because the expected number of pairs of multiple edges in $\X$ is at most
\begin{equation*}
  \binom{M} 2 \frac{k!}{n^k} = O(d^2n^{2-k}) = o(1);
\end{equation*}
and the expected number of loops in $\X$ is
\begin{equation*}
  \E \lambda(\X) \le M\binom k 2 n^{-1} = O(d) = o(n^{1/4}d^{1/2}).
\end{equation*}
  \end{proof}
\noindent Let $\Sc \subset [n]^{nd}$ be the family of all sequences in which every value $i \in [n]$ occurs precisely $d$ times.
Let $\Y = (Y_1, \dots, Y_{nd})$ be a sequence choosen from $\Sc$ uniformly at random. One can equivalently define $\Y$ as a discrete time process determined by the conditional probabilities
  \begin{equation}\label{eq:Ycond}
  \Pc{Y_{t+1} = v}{Y_{1}, \dots, Y_{t}} = \frac{d - \deg_t(v)}{nd - t},  \qquad v = 1, \dots, n, \quad t = 0, \dots, nd-1,
\end{equation}
where
  $$\deg_{t}(v) := |\left\{ 1 \le s \le t : Y_s = v \right\}|.$$
  Assuming $k | (nd)$, we define {a random $d$-regular $k$-multigraph}
  $$\PHnd := \hyp(\Y).$$
  Note that for every $H \in \hcnd$ the number of sequences giving $H$ is the same, namely, $M!(k!)^M$. Therefore $\Hnd$ can be obtained from $\PHnd$ by conditioning on simplicity.

  Probably a more popular way to define $\hh_*^{(k)}(n,d)$ is via the so called \emph{configuration model}, which, for $k =2$, first appeared implicitly in Bender and Canfield \cite{BC} and was given in its explicit form by Bollob\'as \cite{B80} (its generalization to every $k$ is straightforward). A \emph{configuration} is a partition of the set $[n] \times [d]$ into $M$ sets of size $k$, say, $P_1, \dots, P_M$. Then $\PHnd$ is obtained by taking a configuration uniformly at random and mapping every set
$ P_i = \left\{ (v_1, w_1), \dots, (v_k,w_k) \right\}$
to an edge $v_1\dots v_k$.

  The idea of obtaining $\PHnd$ from a random sequence for $k=2$ was used independently Bollob\'as and Frieze \cite{BF} and Chv\'atal \cite{C}.

What makes studying $d$-regular $k$-graphs a bit easier than graphs, at least  for small $d$, is that a.a.s.\ $\Y$ has no multiple edges. However, they usually have a few loops, but, as it turns out, not too many.
Throughout the paper, for $r = 0, 1, \dots$ and $x \in \R$, we use the standard notation $(x)_r := x (x-1) \dots (x-r+1)$.
Recall that $L = n^{1/4}d^{1/2}$.
  \begin{proposition}
    \label{prop:YE}
    If $d \to \infty$, and $d = o(n^{1/2})$, then each of the following statements holds a.a.s.:
    \begin{itemize}
      \item[(i)] $\Y$ has no multiple edges,
      \item[(ii)] $\Y$ has no loop with a vertex of multiplicity at least 3,
      \item[(iii)] $\Y$ has no loop with two vertices of multiplicity at least 2,
      \item[(iv)] $\lambda(\Y) \le L$.
    \end{itemize}
  \end{proposition}
  \begin{proof}
    The first three statements hold because the expected number of undesired objects tends to zero.

(i) The expected number of pairs of multiple edges in $\Y$ is
\begin{multline*}
  \binom{M} 2 \sum_{k_1 + \ldots + k_n = k} \frac{\binom{k}{k_1, \dots, k_n}^2\binom{nd-2k}{d-2k_1,\dots,d-2k_n}}{\binom{nd}{d,\dots,d}}   \le n^2d^2 n^k \frac{ k!^2 d^{2k}}{(nd)_{2k}} = O\left(n^{2-k}d^2\right) = o(1).
\end{multline*}

(ii) The expected number of loops in $\Y$ having a vertex of multiplicity at least $3$ is at most
$$M \frac{\binom k 3 n \binom{nd-3}{d-3, d, \dots, d}}{\binom{nd}{d, \dots, d}} \le nd\frac{k^3nd^3}{(nd)_3} =  O(n^{-1}d) = o(1).$$

(iii) Similarly the expected number of loops in $\Y$ having at least two vertices of multiplicity at least $2$ is at most
$$M\frac{\binom k 2 \binom {k-2} 2 n^2 \binom{nd-4}{d-2, d-2, d,\dots,d}}{\binom{nd}{d, \dots, d}} \le nd\frac{k^4n^2d^4}{(nd)_4} = O(n^{-1}d) = o(1).$$
The statement (iv) follows by Markov's inequality, because 
\begin{equation*}
  \E \lambda(\Y) \le M\frac{\binom k 2 n \binom {nd-2}{d-2, d, \dots, d}}{\binom{nd}{d,\dots,d}} \le nd \frac{k^2nd^2}{(nd)_2} = O(d) = o(n^{1/4}d^{1/2}).
\end{equation*}
  \end{proof}

  In a couple of forthcoming proofs we will need the following concentration inequality (see, e.g., McDiarmid \cite[\S 3.2]{M98}). Let $\Sc_N$ be the set of permutations of $[N]$ and let $\mathbf Z$ be distributed uniformly over $\Sc_N$. Suppose that function $f : \Sc_N \to \R$ satisfies a Lipschitz property, that is, for some $b >0$
   \begin{equation*}
     |f(\mathbf z) - f(\mathbf z')| \le b,
   \end{equation*}
   whenever $\mathbf z'$ can be obtained from $\mathbf z$ by swapping two elements. Then
   \begin{equation}
     \prob{|f(\mathbf Z) - \E f(\mathbf Z)| \ge t } \le 2 \e^{-2t^2/b^2N}, \qquad t \ge 0.
     \label{eq:permconc}
   \end{equation}
We set $r := 2^k + 1$ and $c := 1/(2r+1)$. For the rest of the paper let \begin{equation*}m := \lfloor cM \rfloor.\end{equation*} Color the first $rm$ edges of $\Y$ \emph{red} and the remaining $M - rm$ edges \emph{green}. Define $\Y_{\red} = (Y_1, \dots, Y_{krm})$ and $\Y_{\green} = (Y_{krm+1}, \dots, Y_{nd})$.
Consider a function $\fee : \Sc \to \Z$ defined by
  \begin{equation*}\label{eq:feedef}
    \fee(\y) := \sum_{v=1}^n (\degb(\y;v))_2,
  \end{equation*}
  where
$    \degb(\y;v) := |\left\{ i \in [rkm + 1, kM] : y_i = v\right\}|$
is the \emph{green degree} of $v$. It can be easily checked that
  \begin{equation}\label{eq:feemean}
   \E \fee(\Y) = n (d)_2 \frac{(kM - rkm)_2}{(kM)_2} = \Theta \left( nd^2 \right).
 \end{equation}
 Suppose that sequences $\y, \z \in \Sc$ can be obtained from each other by swapping two coordinates. Since such a swapping affects the green degree of at most two vertices and for every such vertex the green degree changes by at most one, we get
    \begin{equation*}
      |\fee(\y) - \fee(\z)| \le 2\max_{1 \le r \le d} \{ (r)_2 - (r-1)_2 \} = 2 \left( (d)_2 - (d-1)_2 \right) < 4d.
    \end{equation*}
Thus, treating $\Y$ as a permutation of $nd$ elements, \eqref{eq:permconc} implies
    \begin{equation}\label{eq:feemart}
      \prob{|\fee(\Y) - \E \fee(\Y)| \ge x} \le 2\exp\left\{ -\frac{x^2}{8nd^3} \right\}, \qquad x > 0.
    \end{equation}

  \section{Embedding $\Hnm$ into $\PHnd$}\label{sec:embedding}

    A crucial step toward the embedding is to couple the processes $\left( X_t \right)$ and $(Y_t)$, $t = 1, \dots, nd$, in such a way that a.a.s.\  $\X$ and $\Y$ have many edges in common. For this, let $I_1, \dots, I_{nd}$ be an i.i.d.\ sequence of symmetric Bernoulli variables independent of $\X$:
    \begin{equation*}
      \prob{I_t = 0} = \prob{I_t = 1} = 1/2, \qquad t = 1, \dots, nd.
    \end{equation*}
  We define $Y_1, Y_2, \dots$ inductively. Fix $t \ge 0$. Suppose that we have already revealed the values $Y_1, \dots, Y_{t}$.
  If
   \begin{equation}\label{eq:nonneg}
     2\frac{d-\deg_t(v)}{nd - t} - \frac{1}{n} \ge 0 \quad \text{for every } v \in [n],
   \end{equation}
   then generate an auxiliary random variable $Z_{t+1}$ independently of $I_{t+1}$ according to the following distribution (note that the left-hand side of \eqref{eq:nonneg} sums over $v \in [n]$ to 1)
  \begin{equation*}
    \Pc{Z_{t+1} = v}{Y_1, \dots, Y_t} = 2\frac{d-\deg_t(v)}{nd - t} - \frac{1}{n}, \quad v = 1, \dots, n.
  \end{equation*}
  If \eqref{eq:nonneg} holds, set $Y_{t + 1} = I_{t + 1} X_{t + 1} + (1-I_{t + 1}) Z_{t + 1}$. Otherwise generate $Y_{t + 1}$ directly according to the conditional probabilities \eqref{eq:Ycond}.
The distribution of $Z_{t+1}$ is chosen precisely in such a way that \eqref{eq:Ycond} holds for any values of variables $Y_1, \dots, Y_t$, regardless of whether \eqref{eq:nonneg} is satisfied or not. This guarantees that $\Y = (Y_1, \dots, Y_{nd})$ is actually uniformly distributed over $\Sc$.

 The following lemma states that we can embed $\Hnm$ in the red subgraph of $\hh_*^{(k)}(n,d)$.
  \begin{lemma}
    \label{lem:embed}
    For every $k \ge 3$, there is a constant $C > 0$ such that if $d \ge C\log n$ and $d = o(n^{1/2})$, then one can define a joint distribution of $\Hnm$ and $\Y$ in such a way that
    \begin{equation*}
      \Hnm \subset \hyp(\Y_{\red})\quad \text{a.a.s.\ }
    \end{equation*}
  \end{lemma}
  \begin{proof}

Let
\begin{equation*}
  W = \left\{ 0 \le i \le rm - 1 : I_{ki + 1} = \dots = I_{ki + k} = 1 \right\}
\end{equation*}
and let
$\X'$ be the subsequence of $\X$ formed by concatenation of the edges
  $( X_{ki + 1}, \dots, X_{ki + k})$, $i \in W$.
Define the events
\begin{equation*}
  \ac = \left\{\X \text{ has no multiple edges}, \lambda(\X) \le L, |W| \ge m + L  \right\},
\end{equation*}
\begin{equation*}
  \bc = \left\{  \text{inequality \eqref{eq:nonneg} holds for every } v \in [n] \text{ and } t < krm \right\}.
\end{equation*}
Suppose that $\ac$ holds. Then all edges of $\X'$ are distinct and at least $m$ of them are proper. By symmetry, we can take, say, the first $m$ of these edges to form $\Hnm$. If $\ac$ fails, we simply generate $\Hnm$ independently of everything else.

Further, if $\bc$ holds, then for every $i \in W$ we have
\begin{equation*}
  (Y_{ki+1}, \dots, Y_{ki + k}) = (X_{ki+1}, \dots, X_{ki + k}),
\end{equation*}
which is to say that $\hyp(\X')$ is a subgraph of $\hyp(\Y_{\red})$.
Consequently,
\begin{equation*}
  \prob{\Hnm \subset \hyp(\Y_{\red})} \ge \prob{\ac \cap \bc},
\end{equation*}
so it is enough to show that each of the events $\ac$ and $\bc$ holds a.a.s.\

By Proposition \ref{prop:XE}, the first two conditions defining $\ac$ hold a.a.s.  As for the last one, note that $|W| \sim \Bi(rm,2^{-k})$, therefore $\E |W| = (1 + 2^{-k})m$ and $\Var |W| = O(m)$. Since $L = o(m)$, Chebyshev's inequality implies that for $n$ large enough
\begin{equation*}
  \prob{|W| < m + L} \le \frac{\Var |W|}{( 2^{-k}m - L )^2} = O(m^{-1}) = o(1).
\end{equation*}
Concerning the event $\bc$, if for some $t < krm$ and some $v \in [n]$ inequality \eqref{eq:nonneg} does not hold, then $\deg_t(v) > d/2$, and consequently $\deg_{krm}(v) > d/2$. Note that $\deg_{krm}(v)$, $v =1, \dots, n$, are identically distributed hypergeometric random variables. Let $X := \deg_{krm}(1)$. The probability that $\bc$ fails is thus at most
\begin{equation*}
  \prob{ \deg_{krm}(v) > d/2 \text{ for some } v \in [n]} \le n \prob{X > d/2}.
\end{equation*}
We have $\E X = krm/n \le rcd$. Since $c < 1/2r$, applying, say, Theorem 2.10 from~\cite{JLR}, we obtain
\begin{equation*}
  \prob{X > d/2} \le \exp \left\{ - a d \right\} \le \exp\left\{ -aC\log n \right\},
\end{equation*}
for some positive constant $a$. Choosing $C > a^{-1}$ we get $n\prob{X > d/2} = o(1)$, thus concluding the proof.
  \end{proof}

\section{Getting rid of red loops}\label{sec:redloops}
  Let $\ec$ be the family of sequences in $\Sc$ with no multiple edges and containing at most $L$ loops, but no loops of other type than $x_1 x_1 x_2 \dots x_{k-1}$ (up to reordering of vertices), where $x_1, \dots, x_{k-1}$ are distinct.  By Proposition \ref{prop:YE} we have that $\Y \in \ec$ a.a.s.
  Partition~$\ec$ according to the number of loops into sets
  \begin{equation*}
    \ec_l := \left\{ \y \in \ec : \lambda(\y) = l \right\}, \qquad l = 0, \dots, L.
  \end{equation*}
 Let $\gc_l$ be the family of those sequences in $\ec_l$ which contain no red loops. Note that $\gc_0 = \ec_0$ consists precisely of those sequences $\y \in \Sc$ for which $\hyp(\y)$ is simple.

Condition on $\Y \in \ec$ and let $\Y'$ be a sequence obtained from $\Y$ by swapping the red loops of $\Y$ (if any) with a subset of green proper edges chosen uniformly at random.
More formally, let $f_1, \dots, f_r$ be the red loops and $e_1, \dots, e_g$ be the green proper edges of $\Y$ in the order they appear in $\Y$. Pick a set of indices $1 \le i_1 < \dots < i_r \le g$ uniformly at random, and swap $f_j$ with $e_{i_j}$ for $j = 1, \dots, r$, preserving the order of vertices inside the edges. Note that this does not change the underlying $k$-multigraph, that is, $\hyp(\Y) = \hyp(\Y')$.
\begin{proposition}\label{prop:unifG}
  $\Y'$ is uniform on each $\gc_l$, $l = 0, \dots, L$.
\end{proposition}
\begin{proof}
 Fix $l$. Clearly $\Y' \in \gc_l$ if and only if $\Y \in \ec_l$. Also, $\Y$ is uniform on $\ec_l$. For integer $r \in [0, l]$, every $\z \in \gc_l$ can be obtained from the same number (say, $b_r$) of $\y$'s in $\ec_l$ with exactly $r$ red loops. On the other hand, for every $\y$ with exactly $r$ red loops there is the same number (say, $a_r$) of $\z$'s in $\gc_l$ that can be obtained from $\y$. Hence for every $\z \in \gc_l$
  \begin{equation*}
    \Pc{\Y' = \z}{\Y \in \ec_l } = \sum_{r=0}^l \frac{b_r}{a_r|\ec_l|},
  \end{equation*}
which is the same for all $\z \in \gc_l$.
\end{proof}

 The following technical result will be used in the next section. Let

  \begin{equation*}
    \tilde \Sc := \left\{ \y \in \Sc : |\fee(\y) - \E \fee(\Y)| \le n^{3/4}d  \right\}.
  \end{equation*}

 \begin{proposition}\label{prop:probswap}
     If $d = o(n^{1/2})$, then
     $$\prob{ \Y' \in \tilde \Sc} = 1 - o(1).$$
   \end{proposition}
\begin{proof}

    Suppose $\z$ is obtained from $\y$ by swapping a red loop with a green proper edge. This affects the green degree of at most $2k-1$ vertices $v$, and for every such $v$ we have
    \begin{equation*}
      \left| (\degb(\y;v))_2 - (\degb(\z;v))_2 \right| = O(d),
    \end{equation*}
    uniformly for all such $\y, \z$.
    Hence, uniformly
    \begin{equation*}
      |\fee(\Y) - \fee(\Y')|  = O(Ld), \qquad \Y \in \ec.
    \end{equation*}
    By Proposition \ref{prop:YE} we have that $\Y \in \ec$ a.a.s. Hence,
    \begin{multline*}
      \prob{\Y' \notin \tilde \Sc} \le \Prob\Big(|\fee(\Y) - \E \fee(\Y)| > n^{3/4}d - O(Ld) \ \big|\ \Y \in \ec\Big) \\
      \sim \Prob\Big(|\fee(\Y) - \E \fee(\Y)| > n^{3/4}d - O(Ld) \Big).
    \end{multline*}
    Finally, since $d = o(n^{1/2})$, the last probability tends to zero by~\eqref{eq:feemart}.
\end{proof}

\section{Getting rid of green loops}\label{sec:greenloops}

In this section we complete the proof of Theorem \ref{thm:main}, deferring proofs of two technical results to the next section. By Lemma \ref{lem:embed}, which we proved in Section \ref{sec:embedding}, the random $k$\mbox{-}multigraph $\hyp(\Y_{\red})$ contains $\Hnm$ a.a.s.
  Since $\Hnm \subset \hyp(\Y)$ implies that $\Hnm \subset \hyp(\Y')$, it remains to define a procedure, which, a.a.s.\ transforms $\Y'$ (leaving the red edges of $\Y'$ intact) into a random $k$-graph distributed approximately as $\Hnd$.

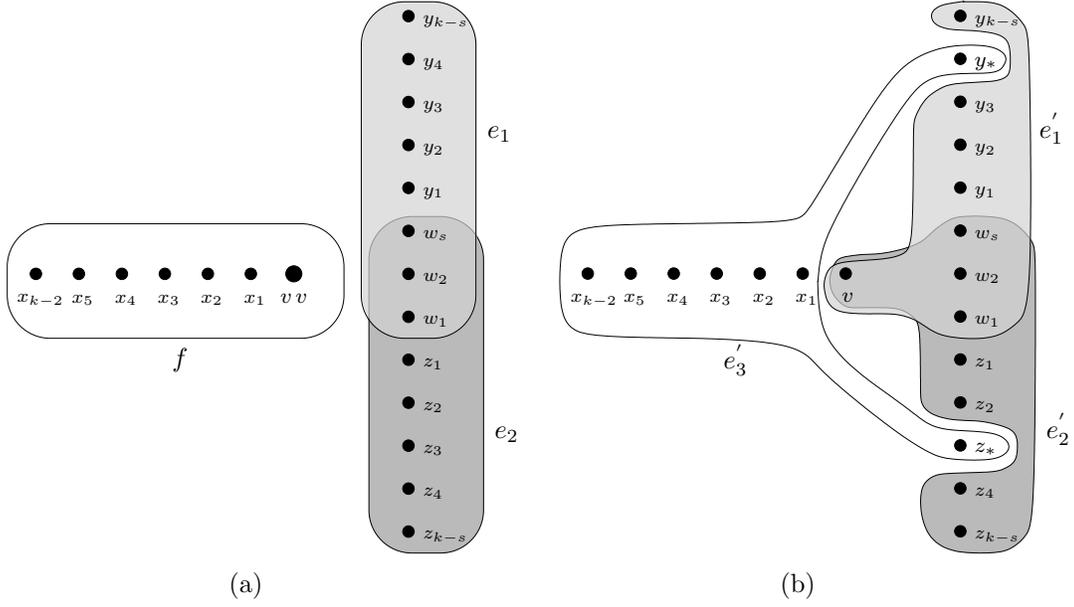
\begin{figure}[h]
\centering
        \subfigure[]
{%
\begin{picture}(0,0)%
\includegraphics{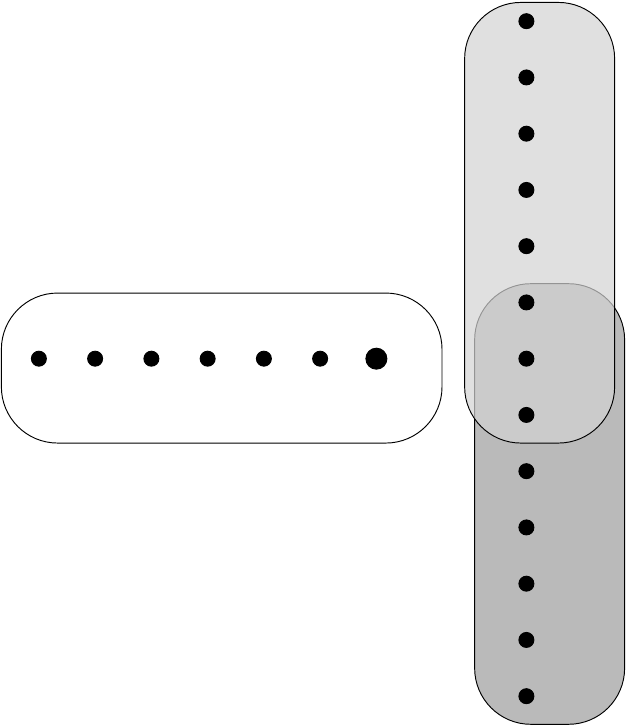}%
\end{picture}%
\setlength{\unitlength}{2368sp}%
\begin{picture}(5127,5799)(5389,-15598)
\put(9751,-12736){\makebox(0,0)[lb]{\smash{{\SetFigFont{7}{8.4}{\rmdefault}{\mddefault}{\updefault}{\color[rgb]{0,0,0}$w_2$}%
}}}}
\put(9751,-11836){\makebox(0,0)[lb]{\smash{{\SetFigFont{7}{8.4}{\rmdefault}{\mddefault}{\updefault}{\color[rgb]{0,0,0}$y_1$}%
}}}}
\put(9751,-11386){\makebox(0,0)[lb]{\smash{{\SetFigFont{7}{8.4}{\rmdefault}{\mddefault}{\updefault}{\color[rgb]{0,0,0}$y_2$}%
}}}}
\put(9751,-12286){\makebox(0,0)[lb]{\smash{{\SetFigFont{7}{8.4}{\rmdefault}{\mddefault}{\updefault}{\color[rgb]{0,0,0}$w_s$}%
}}}}
\put(9751,-10936){\makebox(0,0)[lb]{\smash{{\SetFigFont{7}{8.4}{\rmdefault}{\mddefault}{\updefault}{\color[rgb]{0,0,0}$y_3$}%
}}}}
\put(9751,-10486){\makebox(0,0)[lb]{\smash{{\SetFigFont{7}{8.4}{\rmdefault}{\mddefault}{\updefault}{\color[rgb]{0,0,0}$y_4$}%
}}}}
\put(9751,-13186){\makebox(0,0)[lb]{\smash{{\SetFigFont{7}{8.4}{\rmdefault}{\mddefault}{\updefault}{\color[rgb]{0,0,0}$w_1$}%
}}}}
\put(9751,-10036){\makebox(0,0)[lb]{\smash{{\SetFigFont{7}{8.4}{\rmdefault}{\mddefault}{\updefault}{\color[rgb]{0,0,0}$y_{k-s}$}%
}}}}
\put(9751,-13636){\makebox(0,0)[lb]{\smash{{\SetFigFont{7}{8.4}{\rmdefault}{\mddefault}{\updefault}{\color[rgb]{0,0,0}$z_1$}%
}}}}
\put(9751,-14086){\makebox(0,0)[lb]{\smash{{\SetFigFont{7}{8.4}{\rmdefault}{\mddefault}{\updefault}{\color[rgb]{0,0,0}$z_2$}%
}}}}
\put(9751,-14536){\makebox(0,0)[lb]{\smash{{\SetFigFont{7}{8.4}{\rmdefault}{\mddefault}{\updefault}{\color[rgb]{0,0,0}$z_3$}%
}}}}
\put(9751,-14986){\makebox(0,0)[lb]{\smash{{\SetFigFont{7}{8.4}{\rmdefault}{\mddefault}{\updefault}{\color[rgb]{0,0,0}$z_4$}%
}}}}
\put(9751,-15436){\makebox(0,0)[lb]{\smash{{\SetFigFont{7}{8.4}{\rmdefault}{\mddefault}{\updefault}{\color[rgb]{0,0,0}$z_{k-s}$}%
}}}}
\put(10426,-11236){\makebox(0,0)[lb]{\smash{{\SetFigFont{10}{12.0}{\rmdefault}{\mddefault}{\updefault}{\color[rgb]{0,0,0}$e_1$}%
}}}}
\put(10501,-14386){\makebox(0,0)[lb]{\smash{{\SetFigFont{10}{12.0}{\rmdefault}{\mddefault}{\updefault}{\color[rgb]{0,0,0}$e_2$}%
}}}}
\put(7126,-13636){\makebox(0,0)[lb]{\smash{{\SetFigFont{10}{12.0}{\rmdefault}{\mddefault}{\updefault}{\color[rgb]{0,0,0}$f$}%
}}}}
\put(8259,-12966){\makebox(0,0)[lb]{\smash{{\SetFigFont{8}{9.6}{\rmdefault}{\mddefault}{\updefault}{\color[rgb]{0,0,0}$v\ \!v$}%
}}}}
\put(6076,-12961){\makebox(0,0)[lb]{\smash{{\SetFigFont{7}{8.4}{\rmdefault}{\mddefault}{\updefault}{\color[rgb]{0,0,0}$x_5$}%
}}}}
\put(6526,-12961){\makebox(0,0)[lb]{\smash{{\SetFigFont{7}{8.4}{\rmdefault}{\mddefault}{\updefault}{\color[rgb]{0,0,0}$x_4$}%
}}}}
\put(6976,-12961){\makebox(0,0)[lb]{\smash{{\SetFigFont{7}{8.4}{\rmdefault}{\mddefault}{\updefault}{\color[rgb]{0,0,0}$x_3$}%
}}}}
\put(7426,-12961){\makebox(0,0)[lb]{\smash{{\SetFigFont{7}{8.4}{\rmdefault}{\mddefault}{\updefault}{\color[rgb]{0,0,0}$x_2$}%
}}}}
\put(7876,-12961){\makebox(0,0)[lb]{\smash{{\SetFigFont{7}{8.4}{\rmdefault}{\mddefault}{\updefault}{\color[rgb]{0,0,0}$x_1$}%
}}}}
\put(5502,-12957){\makebox(0,0)[lb]{\smash{{\SetFigFont{7}{8.4}{\rmdefault}{\mddefault}{\updefault}{\color[rgb]{0,0,0}$x_{k-2}$}%
}}}}
\end{picture}%
}%
        \qquad
        \subfigure[]
{%
\begin{picture}(0,0)%
\includegraphics{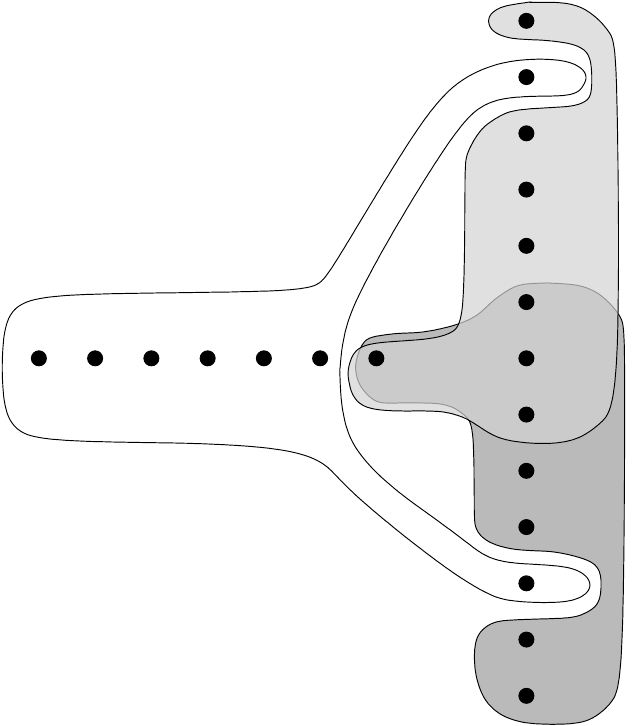}%
\end{picture}%
\setlength{\unitlength}{2368sp}%
\begin{picture}(5127,5804)(5389,-15601)
\put(9751,-12736){\makebox(0,0)[lb]{\smash{{\SetFigFont{7}{8.4}{\rmdefault}{\mddefault}{\updefault}{\color[rgb]{0,0,0}$w_2$}%
}}}}
\put(9751,-11836){\makebox(0,0)[lb]{\smash{{\SetFigFont{7}{8.4}{\rmdefault}{\mddefault}{\updefault}{\color[rgb]{0,0,0}$y_1$}%
}}}}
\put(9751,-11386){\makebox(0,0)[lb]{\smash{{\SetFigFont{7}{8.4}{\rmdefault}{\mddefault}{\updefault}{\color[rgb]{0,0,0}$y_2$}%
}}}}
\put(9751,-12286){\makebox(0,0)[lb]{\smash{{\SetFigFont{7}{8.4}{\rmdefault}{\mddefault}{\updefault}{\color[rgb]{0,0,0}$w_s$}%
}}}}
\put(9751,-10936){\makebox(0,0)[lb]{\smash{{\SetFigFont{7}{8.4}{\rmdefault}{\mddefault}{\updefault}{\color[rgb]{0,0,0}$y_3$}%
}}}}
\put(9751,-13186){\makebox(0,0)[lb]{\smash{{\SetFigFont{7}{8.4}{\rmdefault}{\mddefault}{\updefault}{\color[rgb]{0,0,0}$w_1$}%
}}}}
\put(9751,-10036){\makebox(0,0)[lb]{\smash{{\SetFigFont{7}{8.4}{\rmdefault}{\mddefault}{\updefault}{\color[rgb]{0,0,0}$y_{k-s}$}%
}}}}
\put(9751,-13636){\makebox(0,0)[lb]{\smash{{\SetFigFont{7}{8.4}{\rmdefault}{\mddefault}{\updefault}{\color[rgb]{0,0,0}$z_1$}%
}}}}
\put(9751,-14086){\makebox(0,0)[lb]{\smash{{\SetFigFont{7}{8.4}{\rmdefault}{\mddefault}{\updefault}{\color[rgb]{0,0,0}$z_2$}%
}}}}
\put(9751,-14986){\makebox(0,0)[lb]{\smash{{\SetFigFont{7}{8.4}{\rmdefault}{\mddefault}{\updefault}{\color[rgb]{0,0,0}$z_4$}%
}}}}
\put(9751,-15436){\makebox(0,0)[lb]{\smash{{\SetFigFont{7}{8.4}{\rmdefault}{\mddefault}{\updefault}{\color[rgb]{0,0,0}$z_{k-s}$}%
}}}}
\put(9751,-10486){\makebox(0,0)[lb]{\smash{{\SetFigFont{8}{9.6}{\rmdefault}{\mddefault}{\updefault}{\color[rgb]{0,0,0}$y_*$}%
}}}}
\put(9751,-14536){\makebox(0,0)[lb]{\smash{{\SetFigFont{8}{9.6}{\rmdefault}{\mddefault}{\updefault}{\color[rgb]{0,0,0}$z_*$}%
}}}}
\put(8361,-12961){\makebox(0,0)[lb]{\smash{{\SetFigFont{8}{9.6}{\rmdefault}{\mddefault}{\updefault}{\color[rgb]{0,0,0}$v$}%
}}}}
\put(10501,-14386){\makebox(0,0)[lb]{\smash{{\SetFigFont{10}{12.0}{\rmdefault}{\mddefault}{\updefault}{\color[rgb]{0,0,0}$e_2^{'}$}%
}}}}
\put(10426,-11236){\makebox(0,0)[lb]{\smash{{\SetFigFont{10}{12.0}{\rmdefault}{\mddefault}{\updefault}{\color[rgb]{0,0,0}$e_1^{'}$}%
}}}}
\put(7122,-13658){\makebox(0,0)[lb]{\smash{{\SetFigFont{10}{12.0}{\rmdefault}{\mddefault}{\updefault}{\color[rgb]{0,0,0}$e_3^{'}$}%
}}}}
\put(6076,-12961){\makebox(0,0)[lb]{\smash{{\SetFigFont{7}{8.4}{\rmdefault}{\mddefault}{\updefault}{\color[rgb]{0,0,0}$x_5$}%
}}}}
\put(6526,-12961){\makebox(0,0)[lb]{\smash{{\SetFigFont{7}{8.4}{\rmdefault}{\mddefault}{\updefault}{\color[rgb]{0,0,0}$x_4$}%
}}}}
\put(6976,-12961){\makebox(0,0)[lb]{\smash{{\SetFigFont{7}{8.4}{\rmdefault}{\mddefault}{\updefault}{\color[rgb]{0,0,0}$x_3$}%
}}}}
\put(7426,-12961){\makebox(0,0)[lb]{\smash{{\SetFigFont{7}{8.4}{\rmdefault}{\mddefault}{\updefault}{\color[rgb]{0,0,0}$x_2$}%
}}}}
\put(7876,-12961){\makebox(0,0)[lb]{\smash{{\SetFigFont{7}{8.4}{\rmdefault}{\mddefault}{\updefault}{\color[rgb]{0,0,0}$x_1$}%
}}}}
\put(5520,-12961){\makebox(0,0)[lb]{\smash{{\SetFigFont{7}{8.4}{\rmdefault}{\mddefault}{\updefault}{\color[rgb]{0,0,0}$x_{k-2}$}%
}}}}
\end{picture}%
}%
        \caption{Edges affected by a switching (a) before and (b) after.}
        \label{fig:1}
\end{figure}

For this we define an operation which decreases the number of green loops  one at a time. Two sequences $\y \in \gc_l$, $\z \in \gc_{l-1}$ are said to be \emph{switchable}, if $\z$ can be obtained from $\y$ by the following operation, called a \emph{switching}, which is a generalization (to $k \ge 3$) of a {switching} defined by McKay and Wormald~\cite{MW90a} for $k = 2$. Among the edges of $\y$, choose a loop $f$ and an ordered pair $(e_1, e_2)$ of green proper edges (see Figure \ref{fig:1}a).
 Putting $s = |e_1 \cap e_2|$ and ignoring the order of the vertices inside the edges, one can write
    \begin{equation*}\label{eq:fee}
    f = vvx_1\dots x_{k-2}, \quad e_1 = w_1\dots w_s y_1\dots y_{k-s}, \quad e_2 = w_1\dots w_s z_1 \dots z_{k-s}.
  \end{equation*}
  Loop $f$ contains two copies of $v$, the left one and the right one (with respect to their order in the sequence $\y$). Select vertices $y_* \in \left\{ y_1, \dots, y_{k-s} \right\}$ and $z_* \in \{z_1, \dots, z_{k-s}\}$, and swap $y_*$ with the left copy of $v$ and $z_*$ with the right one. The effect of switching is that $f, e_1$, and $e_2$ are replaced by three proper edges (see Figure \ref{fig:1}b):
  \begin{equation*}
    e'_1 = e_1 \cup \{v\} - \{y_*\}, \qquad e'_2 = e_2 \cup \{v\} - \{z_*\}, \qquad e'_3 = f \cup \{y_*,z_*\}- \{v,v\}.
  \end{equation*}

  A \emph{backward switching} is the reverse operation that reconstructs $\y \in \gc_l$ from $\z \in \gc_{l-1}$. It is performed by choosing a vertex $v$, an ordered pair of green proper edges $e'_1, e'_2$ containing $v$, one more green proper edge $e_3'$, choosing a pair of vertices $y, z \in e_3'$, and swapping $y$ with the copy of $v$ in $e_1$ and $z$ with the one in $e_2$.

  Note that, given $f, e_1, e_2$, not every choice of $y_*, z_*$ defines a forward switching, due to possible creation of new loops or multiple edges. We say that the choices of $y_*, z_*$ which do define a switching are \emph{admissible}.
  Similarly a choice of $y, z$ is \emph{admissible} with respect to $v, e_1', e_2'$, and $e_3'$ if it defines a backward switching.

Given $\y \in \gc_l$, let $F(\y)$ and $B(\y)$ be the number of ways to perform forward switching and backward switching, respectively.

Let $\sw$ denote a (random) operation which, given $\y \in \gc_l$, applies to it a forward switching, chosen uniformly at random from the $F(\y)$ possibilities. Let $\Y'' \in \gc_0$ be the sequence obtained from $\Y'$ by applying $\sw$ until there are no loops left, namely, $\lambda(\Y')$ times. Suppose for a moment that for every $l$ and $\y \in \gc_l$ functions $F(\y)$ and $B(\y)$ depend on $l$, but not on the actual choice of $\y$. If this were true, then, as one could easily show, $\Y''$ would be uniform over $\gc_0$. As we will see, we are not far from this idealized setting, because Proposition~\ref{prop:FB}(a) below implies that $F(\y)$ is essentially proportional to $l = \lambda(\y)$. On the other hand, Proposition \ref{prop:FB}(b) shows that $B(\y)$ depends on a more complicated parameter of $\y$, namely on  $\fee(\y)$ defined in Section \ref{sec:prelim}.

  To make $B(\y)$ essentially independent of $\y$, we will apply switchings not to every element of $\gc_0 \cup \dots \cup \gc_L$, but to a slightly smaller subfamily. Let
$$\tilde \gc_l := \gc_l \cap \tilde \Sc, \qquad l = 0, \dots, L,$$
where $\tilde \Sc$ has been defined in the previous section.

  We condition on $\Y' \in \tilde \Sc$ and deterministically map $\Y''$ to a simple $k$-graph
  \begin{equation*}
    \THnd := \hyp(\Y'').
  \end{equation*}
  Note that switching does not affect the green degrees, and thus does not change the value of $\fee$. Therefore, if one applies a forward or backward switching to a sequence $\y \in \tilde \Sc$, the resulting sequence is also in $\tilde \Sc$.  Moreover, Proposition \ref{prop:probswap}   shows that by restricting $\Y'$ to $\tilde \Sc$, we do not exclude many sequences.

  The following proposition quantifies the amount by which a single application of $\sw$ distorts the uniformity of $\Y'$.
  \begin{proposition}\label{prop:FB}
  If $1 \le d = o(n^{1/2})$, then
  \begin{itemize}
    \item[$\mathrm {(a)}$] for $\y \in \gc_l, 0 < l \le L$,
      \begin{equation*}
	k^2l(M - rm)^2 \left(1 - O\left(\frac{L + d^2}{M}\right)\right) \le F(\y) \le k^2l(M-rm)^2,
      \end{equation*}
    \item[$\mathrm {(b)}$]  for $\y \in \gc_l$, $0 \le 1 < L$,
      \begin{multline*}
	\binom k 2 (\fee(\y) - 2kLd)(M - rm)\left(1 - O\left(\frac{L + d^2}{M} \right)\right) \\
	\leq B(\y) \leq \binom k 2 \fee(\y)(M - rm).
      \end{multline*}
    \item[$\mathrm {(b')}$] for $\y \in \tilde \gc_l$, $0 \le 1 < L$
      \begin{multline*}
	\binom k 2 \E\fee(\Y)(M - rm)\left(1 - O\left(\frac{n^{3/4}d}{\E \fee(\Y)} + \frac{L + d^2}{M} \right) \right) \\
	\leq B(\y)\leq \binom k 2 \E\fee(\Y)(M - rm)\left(1 + O\left(\frac{n^{3/4}d}{\E \fee(\Y)}\right)\right).
      \end{multline*}
  \end{itemize}
\end{proposition}

Finally, we need to show that the final step of the procedure, that is, the mapping of $\Y''$ to $\hyp(\Y'')$ has negligible influence on the uniformity of the distribution. For this, set
\begin{equation*}
  P_H := |\hyp^{-1}(H) \cap \tilde \gc_0| = \left|\left\{ \y \in \tilde \gc_0 : \hyp(\y) = H \right\}\right|, \qquad H \in \hcnd.
\end{equation*}
   \begin{proposition}\label{prop:fminmax}
     If $d = o(n^{1/2})$, then uniformly for every $H \in \hcnd$
  \begin{equation*}
    (1- o(1))M!(k!)^M \le P_H \le M!(k!)^M .
  \end{equation*}
\end{proposition}

\noindent Proofs of Propositions \ref{prop:FB} and \ref{prop:fminmax} can be found in Section \ref{rp}.

\begin{lemma}
  \label{lem:almostuniform}
  There is a sequence $\eps_n = o(1)$ such that for every $H \in \hcnd$
  \begin{equation*}
    \prob{\THnd = H} = (1 \pm \eps_n)|\hcnd|^{-1}.
  \end{equation*}
\end{lemma}
\begin{proof}
Clearly it is enough to show that for some function $p = p(n,l)$ we have
  \begin{equation}\label{eq:Hcond}
  \Pc{\THnd = H}{\Y' \in \tilde \gc_l} = (1 + o(1))p(n,l)
\end{equation}
uniformly for $l \le L$ and $H \in \hcnd$.
Indeed,
  \begin{equation*}
    \prob{\THnd = H} = \sum_{l = 0}^L \Pc{\THnd = H}{\Y' \in \tilde \gc_l}\prob{\Y' \in \tilde \gc_l} = (1 + o(1))p(n),
  \end{equation*}
  where $p(n) := \sum_l p(n,l) \Prob(\Y' \in \tilde \gc_l)$ is independent of $H$.

  Let $F_l = k^2l(M-rm)^2$ and $B = \binom k 2 \E \fee(\Y) (M-rm)$ be the asymptotic values of the bounds in Proposition \ref{prop:FB}, (a) and (b'), respectively.

  By Proposition \ref{prop:unifG}, we can treat $\Y'$ as a uniformly chosen element of $\tilde \gc_l = \gc_l \cap \tilde \Sc$. Every realization of $l$ switchings that generate $\Y''$ produces a \emph{trajectory}
  \begin{equation*}
    (\y^{(l)},  \dots, \y^{(0)}) \in \tilde \gc_l \times \dots \times \tilde \gc_0,
  \end{equation*}
  where $\y^{(k)}$ is switchable with $\y^{(k-1)}$ for $k = 1, \dots, l$. The probability that a particular such trajectory occurs is
  \begin{multline}\label{eq:trajprob}
    \frac{1}{|\tilde \gc_l|F(\y^{(l)}) \dots F(\y^{(1)})} =
    \left( 1 + O\left( \frac{L + d^2}{M} \right) \right)^l |\tilde \gc_l|^{-1} \prod_{i=1}^l F_i^{-1} \\
    = (1+o(1)) |\tilde \gc_l|^{-1} \prod_{i=1}^l F_i^{-1},
  \end{multline}
  the first equality following from Proposition \ref{prop:FB}.

  On the other hand, by Propositions \ref{prop:FB} and \ref{prop:fminmax} the number of trajectories that lead to a particular $H \in \hcnd$ is
  \begin{equation}\label{eq:trajcount}
    P_H B^l \left(1+ O\left(\frac{n^{3/4}d}{\E \fee(\Y)} + \frac{L + d^2}{M}\right)\right)^l= (1 + o(1)) M! (k!)^M B^l,
  \end{equation}
  because $\E \fee(\Y) = \Theta(nd^2)$ by \eqref{eq:feemean}.
  Now the estimate \eqref{eq:Hcond} with
  \begin{equation*}
    p(n,l) = M!(k!)^M B^l |\tilde \gc_l|^{-1}\prod_{i=1}^l F_i^{-1}
  \end{equation*}
  follows by multiplication of \eqref{eq:trajprob} and \eqref{eq:trajcount}.
\end{proof}

\begin{proof}[Proof of Theorem \ref{thm:main}]
  Let $\mu$ be a uniform distribution over $\hcnd$ and $\nu$ be the distribution of $\THnd$, that is
  \begin{equation*}
    \mu(H) = |\hcnd|^{-1}, \qquad \nu(H) = \prob{\THnd = H}, \qquad H \in \hcnd.
  \end{equation*}
  By Lemma \ref{lem:almostuniform} the \emph{total variation distance} between the measures $\mu$ and $\nu$ is
  \begin{equation*}
    \tvd(\mu,\nu) := \frac 1 2 \sum_{H \in \hcnd} |\mu(H) - \nu(H)| \le \frac 1 2 \sum_H \eps_n \mu(H) = o(1).
  \end{equation*}
  Therefore a standard fact from probability theory (see, e.g., \cite[p. 254]{BHJ}) implies that there is a joint distribution of $\THnd$ and $\Hnd$ such that
  \begin{equation}\label{eq:hhH}
    \THnd = \Hnd \qquad \text{a.a.s.\ }
  \end{equation}
  By definition of $\THnd$, if $\Hnm \subset \hyp(\Y_{\red})$, then $\Hnm \subset \THnd$. Therefore, Theorem \ref{thm:main} follows by Lemma \ref{lem:embed} and Proposition \ref{prop:probswap}.
\end{proof}

\section{Remaining proofs}\label{rp}

\begin{proof}[Proof of Proposition \ref{prop:FB}]
  (a) The upper bound follows from the fact that after we choose (in one of at most $l(M-rm)^2$ ways) a loop and two green edges, we have at most $k^2$ admissible choices of vertices $y_*$ and $z_*$.

  We say that two edges $e', e''$ of a $k$-graph are \emph{distant} from each other if they do not intersect and there is no third edge $e'''$ that intersects both $e'$ and $e''$. Note that for any edge $e$ there are at most $k^2d^2$ edges \emph{not} distant from $e$.

  For the lower bound, let us estimate the number of triples $(f,e_1,e_2)$ for which we have exactly $k^2$ admissible choices of $y_*, z_*$. For this it is sufficient that $e_1 \cap e_2 = \emptyset$ and both $e_1, e_2$ are distant from $f$ in $\hyp(\y)$. Given $f$, we can choose such $e_1$ in at least $M - rm - l-k^2d^2 = (M- rm)(1 - O\left(  (L + d^2)/M\right))$ ways and then choose such $e_2$ in at least $M-rm - l - k^2d^2 - kd = (M-rm)(1-O\left( (L+d^2)/M \right))$ ways. Hence the lower bound.

  (b) We can choose a vertex $v \in [n]$ and an ordered pair of edges $e_1', e_2'$ containing $v$ in at most $\fee(\y)$ ways and then choose $e_3'$ in at most $M - rm$ ways. Number of admissible choices of vertices $y, z \in e_3'$ is at most $\binom k 2$, which gives the upper bound.

  For the lower bound, we estimate the number of quadruples $v, e_1'$, $e_2', e_3'$ for which there are exactly $\binom k 2$ admissible choices of $y,z$. For this it is sufficient that $e_3'$ is distant from both $e_1'$ and $e_2'$ in $\hyp(\y)$. The number of ways to choose $v, e_1', e_2'$ is exactly
  \begin{equation}\label{eq:sumdegs}
    \sum_{v \in [n]} \left( \degb'(\y;v) \right)_2,
  \end{equation}
  where $\degb'(\y;v)$ is the number of green \emph{proper} edges containing vertex $v$. It is obvious that \eqref{eq:sumdegs} is at most $\fee(\y)$ and, as one can easily see, at least $\fee(\y) - 2kLd$. The lower bound now follows, since, given $v, e_1', e_2'$, we can choose $e_3'$ in at least $M - rm - l - 4k^2d^2 = (M-rm)(1 - O\left( (L + d^2)/M \right)$ ways.

  (b$'$) Immediate from (b) and the definition of $\tilde \gc_l$.
\end{proof}

\begin{proof}[Proof of Proposition \ref{prop:fminmax}]
  The upper bound is just $|\hyp^{-1}(H)|$. For the lower bound, we let $\Y|_H$ be a sequence chosen uniformly at random from  $\hyp^{-1}(H)$ and show that the probabilities
  \begin{equation*}
    \prob{|\fee(\Y|_H) - \E \fee(\Y)| > n^{3/4}d}, \qquad H \in \hcnd,
  \end{equation*}
  uniformly tend to zero.
  Since $\fee$ does not depend on the order of vertices inside the edges of $\Y$, we can treat $\Y|_H$ as a random permutation of the $M$ edges of $H$, which we denote by $e_1, \dots, e_M$. Since $H$ is simple, we have
  \begin{equation*}
    \fee(\Y|_H) = \sum_{v \in [n]} \sum_{\substack{e_i, e_j \ni v \\ i \neq j}} \I_{\{e_i, e_j \text{ are green in } \Y|_H \}},
  \end{equation*}
  whence
  \begin{equation*}
    \E \fee \left( \Y|_H \right) = n(d)_2 \frac{(M-rm)_2}{(M)_2}.
  \end{equation*}
  Therefore  \eqref{eq:feemean} and simple calculations yield
  \begin{equation*}
    \E\fee(\Y|_H) - \E \fee(\Y) = O(nd^2M^{-1}) = O(d).
  \end{equation*}
  Further, if $\y,\z \in \hyp^{-1}(H)$ and $\z$ can be obtained from $\y$ by swapping two edges, then
  \begin{equation*}
    |\fee(\y) - \fee(\z)| = O(d),
  \end{equation*}
  uniformly for all such $\y$ and $\z$.
  Therefore \eqref{eq:permconc} applies to $f = \fee$ with $N = M$ and $b = O(d)$. To sum up,
  \begin{multline*}
    \prob{|\fee(\Y|_H) - \E \fee(\Y)| > n^{3/4}d }
    \le \prob{|\fee(\Y|_H) - \E \fee(\Y|_H)| > n^{3/4}d - O(d)} \\
    \le 2 \exp \left\{ - \frac{\left( n^{3/4}d - O(d) \right)^2}{O\left( Md^2 \right)} \right\}= o(1),
  \end{multline*}
  the equality following from the assumption $d = o(n^{1/2})$.
\end{proof}

 \section{Concluding Remarks}
  \begin{remark}
    Theorem \ref{thm:main} is closely related to a result of Kim and Vu \cite{KV04}, who proved, for~$d$ growing faster than $\log n$ but slower than $n^{1/3}/\log^2 n$, that there is a joint distribution of $\hh^{(2)}(n,p)$ and $\hh^{(2)}(n,d)$ with $p$ satisfying $p \sim d/n$ so that
    \begin{equation}\label{eq:KV}
      \hh^{(2)}(n,p) \subset \hh^{(2)}(n,d) \qquad \text{a.a.s.\ }
    \end{equation}
    It is known (see, e.g., \cite{B01}) that $\hh^{(2)}(n,p)$ is a.a.s.\  Hamiltonian, when the expected degree $(n-1)p$ grows faster than $\log n$. Therefore \eqref{eq:KV} implies an analogue of Corollary~\ref{cor:ham} for graphs.
  \end{remark}

 \begin{remark}
   In \cite{DFRS} the authors used the same switching as in the present paper to count $d$-regular $k$-graphs approximately for $k \ge3$ and $1 \le d = o(n^{1/2})$ as well as for $k \ge 4$ and $d = o(n)$. The application of the technique is somewhat easier there, because there is no need to preserve the red edges. The restriction $d = o(n^{1/2})$ that appears in Theorem \ref{thm:main} has also a natural meaning in \cite{DFRS}, since the counting formula there gives the asymptotics of the probability $p_{n,d} := \Prob(\PHnd \text{ is simple})$ for $d = o(n^{1/2})$, while for $k \ge 4$ and $n^{1/2} \le d = o(n)$ it just gives the asymptotics of $\log p_{n,d}$.
 \end{remark}
 \begin{remark}\label{rem:lower}
   The lower bound on $d$ in Theorem \ref{thm:main} is necessary because the second moment method applied to $\Hnp$ (cf. Theorem 3.1(ii) in \cite{B80}) and asymptotic equivalence of $\Hnp$ and $\Hnm$ yields that for $d = o(\log n)$ and $m \sim cM$ there is a sequence $\Delta = \Delta(n)$ such that $d = o(\Delta)$ and the maximum degree $\Hnm$ is at least $\Delta$ a.a.s.\
 \end{remark}
 \begin{remark}
   For $d$ greater than $\log n$, however, the degree sequence of $\Hnp$ is closely concentrated around the expected degree. Therefore it is plausible that Theorem \ref{thm:main} can be extended to $d$ greater than $n^{1/2}$. However, $n^{1/2}$ seems to be an obstacle which cannot be overcome without a proper refinement of our proof.
 \end{remark}
 \begin{remark}
  In view of Remark \ref{rem:lower}, our approach cannot be extended to $d = O(\log n)$. Nevertheless, we believe that the following extension of Corollary~\ref{cor:ham} is valid.
  \begin{conjecture}
    For every $k\ge 3$ there is a constant $d_0=d_0(k)$ such that for any $d \ge d_0$,
    \begin{equation*}
      \hh^{(k)}(n,d) \text{ contains a loose Hamilton cycle a.a.s.}
    \end{equation*}
  \end{conjecture}
  \noindent
   Recall that Robinson and Wormald \cite{RW92, RW94} proved for $k=2$ that as far as fixed $d$ is considered, it suffices to take $d \ge 3$. Their approach is based on a very careful analysis of variance of a random variable counting the number of Hamilton cycles in the configuration model. Unfortunately, for $k\ge 3$ similar computations become extremely complicated and involved, discouraging one from taking this approach.
 \end{remark}
 \begin{remark}
In this paper, we were concerned only with loose cycles. One can also consider a more general problem. 
Define an {\em $\ell$-overlapping  cycle} as a $k$-uniform hypergraph in which,  for some cyclic ordering
of its vertices, every edge consists of $k$ consecutive vertices, and every two consecutive edges
(in the natural ordering of the edges induced by the ordering of the vertices) share exactly $\ell$ vertices. (Clearly, $\ell=1$ corresponds to loose cycles.)
The thresholds for the existence of $\ell$-overlapping Hamilton cycles in $\Hnp$ have been recently obtained in~\cite{DF2}. However, proving similar results for $\hh^{(k)}(n,d)$ and arbitrary $\ell\ge 2$ seem to be hard. Based on results from~\cite{DF2} we believe that the following is true.

  \begin{conjecture}
    For every $k > \ell \ge 2$ if $d \gg n^{\ell-1}$, then
    \begin{equation*}
      \hh^{(k)}(n,d) \text{ contains an $\ell$-overlapping Hamilton cycle a.a.s.}
    \end{equation*}
  \end{conjecture}

 \end{remark}
\bibliographystyle{abbrv}

  \end{document}